\documentclass[12pt,reqno]{amsart}
\usepackage{amssymb}
\usepackage{amsmath}
\usepackage{enumitem}
\usepackage{mathrsfs}
\usepackage[all]{xy}
\usepackage{hyperref}
\usepackage{marginnote}

\usepackage[margin=1.25in]{geometry}

\RequirePackage{mathrsfs} 
\usepackage[OT2,T1]{fontenc}
\DeclareSymbolFont{cyrletters}{OT2}{wncyr}{m}{n}
\DeclareMathSymbol{\Sha}{\mathalpha}{cyrletters}{"58}

\newtheorem{theorem}{Theorem}[section]

\newtheorem{proposition}[theorem]{Proposition}
\newtheorem{corollary}[theorem]{Corollary}

\theoremstyle{definition}
\newtheorem*{ack}{Acknowledgements}
\newtheorem*{con}{Conventions}
\newtheorem{remark}[theorem]{Remark}

\newtheorem{definition}[theorem]{Definition}

\numberwithin{equation}{section} \numberwithin{figure}{section}

\DeclareMathOperator{\Pic}{Pic} 
 
\DeclareMathOperator{\Aut}{Aut} 

\DeclareMathOperator{\Spec}{Spec}

\DeclareMathOperator{\tr}{tr}

\DeclareMathOperator{\chr}{char}

\DeclareMathOperator{\Lin}{Lin}

\newcommand\PP{\mathbb{P}}

\newcommand\ZZ{\mathbb{Z}}

\newcommand\QQ{\mathbb{Q}}

\newcommand\CC{\mathbb{C}}

\newcommand\OO{\mathcal{O}}

\newcommand{\et}{\textrm{\'{e}t}}

\title[Smooth hypersurfaces and level structure]{The moduli of smooth hypersurfaces with level structure}

\author{A. Javanpeykar}
\address{A. Javanpeykar \\
Institut f\"{u}r Mathematik\\
Johannes Gutenberg-Universit\"{a}t Mainz\\
Staudingerweg 9, 55099 Mainz\\
Germany.}
\email{peykar@uni-mainz.de}

\author{D. Loughran}
\address{D. Loughran \\ School of Mathematics \\
University of Manchester \\
Oxford Road \\
Manchester\\
M13 9PL\\
UK.}
\email{daniel.loughran@manchester.ac.uk}



\begin{document}
 
\begin{abstract}
	We construct the moduli space of smooth hypersurfaces with level $N$ structure over $\ZZ[1/N]$. As an application we show that, for $N$ large enough, the stack of smooth hypersurfaces over $\ZZ[1/N]$ is uniformisable by a smooth affine scheme. To prove our results, we use the Lefschetz trace formula to show that automorphisms of smooth hypersurfaces act faithfully on their cohomology. We also prove a global Torelli theorem for smooth cubic threefolds over fields of odd characteristic.
\end{abstract}

\maketitle

\thispagestyle{empty}

\section{Introduction}

The moduli of smooth proper curves of genus $g$ with $g\geq 2$, or principally polarized abelian schemes of fixed dimension, or polarized K3 surfaces of fixed degree are smooth finite type separated Deligne-Mumford stacks over $\mathbb Z$. All these stacks admit \emph{level structures} \cite{MoretBailly, Rizov, Serre}. Such structures are usually introduced to help rigidify the moduli problem and lead to interesting theory and applications \cite{PoppI}.

The aim of this note is to construct a moduli stack of smooth hypersurfaces 
with level structure. We will define a level $N$ structure  on a smooth hypersurface to be a trivialization of its cohomology with $\mathbb Z/N\ZZ$-coefficients (see \S\ref{sec:level} for details). Key to our construction is the following result on the action of an automorphism of a smooth hypersurface on its cohomology.  
 
\begin{theorem}\label{thm:ind_auts}  Let $d\geq 3$ and $n\geq 1$ be integers with $(d,n)\neq (3,1)$. Let $k$ be a field and let $\ell$ be a prime number which is invertible in $k$. Let $X$ be a smooth hypersurface of degree $d$ in $\mathbb P^{n+1}_k$, and let $\sigma \in \mathrm{Aut}(X)$ be non-trivial. If $\mathrm{char}(k) = 0$ or the order of $\sigma$ is coprime to $\mathrm{char}(k)$, then $\sigma$ acts non-trivially on $\mathrm{H}^n(X_{\bar k,\et},\QQ_\ell)$. 
\end{theorem}

The question of whether the automorphism group of a variety acts faithfully on its cohomology has been investigated for other families of varieties, such as Enriques surfaces  \cite{MukaiNamikawa},  hyperk\"ahler varieties \cite[Prop.~9]{Beauville2}, \cite[\S3]{Boissiere}, \cite[\S2.4]{Rizov}, and some surfaces of general type \cite{C3, Peters}.

Let $d\geq 3$ and let $n\geq 1$ be integers with $(d,n) \neq (3,1)$. Let $\mathcal C_{d,n}$ be the stack of smooth hypersurfaces of degree $d$ in $\mathbb P^{n+1}$ (see \S \ref{sec:level}). Mumford has shown that  $\mathcal C_{d,n}$ is  a smooth finite type separated Deligne-Mumford stack over $\mathbb Z$ whose coarse moduli space is an affine scheme. 
For an integer $N\geq 1$, we will define an algebraic stack $\smash{\mathcal C_{d,n}^{[N]}}$  over $\ZZ[1/N]$ parametrizing smooth hypersurfaces with a level $N$ structure, and our main result (Theorem \ref{thm:level}) is that it is representable by a smooth affine  scheme when $N$ is large enough.

To state our result, let $a_{d,n} $ be the product of all primes $p$  for which there exist an algebraically closed field $k$ and a smooth hypersurface $X$ of degree $d$ in $\PP^{n+1}_k$ with a linear automorphism of order $p$ (Definition \ref{defn:adn}). 
We will show in Section \ref{sec:level}   that $a_{d,n}$ is a well-defined positive integer. It divides the product of all orders of linear automorphism groups of smooth hypersurfaces of degree $d$ in $\mathbb P^{n+1}_{k}$, as $k$ runs over all algebraically closed fields.     In particular,   Theorem \ref{thm:ind_auts} shows that any non-trivial automorphism $\sigma$ acts faithfully on the cohomology of $X$ as long as the characteristic of $k$ is coprime to $a_{d,n}$.  
 
\begin{theorem}\label{theorem2}
Let $d\geq 3$ and let $n\geq 1$ be integers. For all $N\geq 3$ coprime to $a_{d,n}$, the stack $\mathcal C_{d,n,\ZZ[1/N]}$ is uniformisable by a smooth affine scheme $U$   over $\mathbb Z[1/N]$.
\end{theorem}

Here, following Noohi \cite[Def.~6.1]{Noohi}, we say that an algebraic stack $X$ is uniformisable  if there exist an algebraic space $U$ and a finite \'etale morphism $U\to X$.  

Versions of Theorem \ref{thm:ind_auts} and Theorem \ref{theorem2} were obtained by the authors in \cite{JL} for other complete intersections in projective space
\emph{over $\CC$}. These constructions used the infinitesimal Torelli theorem for smooth complete intersections and spreading out arguments, hence only showed  
that $\mathcal C_{d,n}$ becomes uniformisable after base-changing to some non-explicit arithmetic curve $B = \Spec \mathcal O_K[S^{-1}]$ (cf.~the proof of \cite[Prop.~2.12]{JL}). The significance
of Theorem \ref{theorem2} is that we obtain an uniformisation over $\ZZ[1/N]$ for sufficiently large $N$.

To illustrate how one can use the faithfulness of the action of the automorphism group on \'etale cohomology, we finish with an application to the global Torelli problem for smooth cubic threefolds (here it is already known that any automorphism acts faithfully on cohomology, by work of Pan \cite[Thm.~1.2]{Pan}).  A famous theorem of Clemens and Griffiths \cite{CG72} states that any cubic threefold over $\CC$ is uniquely determined by its \emph{intermediate Jacobian}. The intermediate Jacobian is usually defined via transcendental techniques, however it has been known for some time that this theory can be made to work for cubic threefolds over other fields \cite{BS67, Del72}, with a definitive construction over schemes of characteristic not equal to $2$ being given recently by Achter in \cite{Ach}. In \cite{JL} we obtained an extension of Clemens and Griffiths' Torelli theorem to arbitrary fields of characteristic $0$. We will use  the Torelli theorem of Beauville \cite{Beau82} over algebraically closed fields of characteristic not equal to $2$   to obtain the following.

\begin{theorem}\label{thm:Torelli}
Let $k$ be a field of characteristic not equal to $2$ and let $X_1,X_2$ be smooth cubic threefolds over $k$. If the intermediate Jacobians $J(X_1)$ and $J(X_2)$ of $X_1$ and $X_2$ are isomorphic as principally polarised abelian varieties over $k$, then $X_1 \cong X_2$.
\end{theorem}

Whilst finishing this paper, we learned of the recent results of  Chen-Pan-Zhang  \cite{CPZ15}. 
 Here the aforementioned methods relying on infinitesimal Torelli are used in a way similar to \cite{JL} to prove a version of Theorem \ref{thm:ind_auts} for complete intersections, including some cases of positive characteristic for which the infinitesimal Torelli theorem holds. Our paper deals with tame automorphisms in arbitrary characteristic and gives a new and completely different elementary proof using the Lefschetz trace formula.

\begin{ack} 
We thank Olivier Benoist for discussions on level structures. We thank  Benjamin Bakker, Bas Edixhoven,  Dragos Fratila, Frank Gounelas, David Holmes,   Robin de Jong, Max Lieblich, Martin Olsson, Jan Steffen M\"{u}ller and David Rydh for helpful conversations. We are grateful to the referee for some useful comments. We also thank Dingxin Zhang for pointing out a mistake in an earlier version of this paper. The first-named author gratefully acknowledges
the support of SFB/Transregio 45.
\end{ack}

\begin{con}

For $X$ a scheme over a field $k$, we let $\Aut(X)$ be the group of automorphisms of $X$ over $k$.
If $B$ is a  scheme and $N\neq 0$ is an integer, then we write $B_{\ZZ[1/N]}$ for $B \times_{\Spec \ZZ} \Spec \ZZ[1/N]$.

\end{con}

\section{Tame linear automorphisms act faithfully on cohomology}\label{sec:aut}
Let $k$ be a field and let $\ell$ be a prime number which is invertible in $k$.
Let $X$ be a smooth hypersurface of degree $d$ in $\mathbb P^{n+1}_k$. Recall that, by the Lefschetz hyperplane section theorem, the cohomology ring $\mathrm{H}^*(X_{\bar k,\et},\ZZ_\ell)$ is torsion free, and that for $i \neq n$, the group $\mathrm{H}^i(X_{\bar k,\et},\ZZ_\ell)$ is trivial if $i$ is odd and isomorphic to $\ZZ_\ell$ if $i$ is even. In particular, the only ``interesting'' cohomology group is $\mathrm{H}^n(X_{\bar k,\et},\ZZ_\ell)$.

Denote by $\Lin(X)$ the group of \emph{linear automorphisms} of the $k$-scheme $X$, i.e.~those automorphism which are induced by an automorphism of the ambient projective space. When $d\geq 3$, it is known that $\Lin(X)$ is finite \cite[Thm.~3.1]{Ben13}. Moreover, if $(d,n) \neq (3,1),(4,2)$,
then $\Lin(X) = \Aut(X)$ \cite[Thm.~3.1]{Ben13}.

\begin{proposition}\label{prop:ind_auts}  Let $d\geq 3$ and $n\geq 1$ be integers such that $(d,n) \neq (3,1)$. Let $k$ be a field and let $\ell$ be a prime number which is invertible in $k$.
	Let $X$ be a smooth hypersurface of degree $d$ in $\mathbb P^{n+1}_k$ and let $\sigma \in \Lin(X)$ be non-trivial. If $\mathrm{char}(k) = 0$ or the order of $\sigma$ is coprime to $\mathrm{char}(k)$, then $\sigma$ acts non-trivially on $\mathrm{H}^n(X_{\bar k,\et},\QQ_\ell)$.
\end{proposition}

 \begin{proof}
	To prove the result, we may assume that $k$ is algebraically closed.
	We may also assume that $\mathrm{char}(k)> 0$, as when $\mathrm{char}(k)= 0$ the result is
	a special case of \cite[Prop.~2.16]{JL}.
	
	Let $\sigma \in \Lin(X)$ be non-trivial.
	As the order of $\sigma$ is coprime to $\mathrm{char}(k)$, 
	we may decompose $\mathrm{H}^0(X, \OO_X(1))$
	into a direct sum of $t$ eigenspaces $\mathrm{H}^0(X, \OO_X(1)) = \bigoplus_{i=1}^t V_i$
	with dimensions $m_i \geq 1$. Note that 
	\begin{equation} \label{eqn:m_i}
		m_1 +\cdots + m_t = n+2. 
	\end{equation} 
	The fixed locus $X^\sigma$ of $\sigma$ acting on $X$ is
	the disjoint union
	$\bigsqcup_{i=1}^t X_i$, where $ X_i = X \cap \PP(V_i)$
	and $\PP(V_i)\subset \PP^{n+1}_k$ denotes the corresponding projective subspace.
	As the order of $\sigma$ is coprime to $\mathrm{char}(k)$, the fixed locus
	$X^\sigma$ is smooth \cite[Prop.~A.8.10]{CGP}, hence each $X_i$ is smooth.
	An elementary argument, given by  choosing a basis for each $V_i$, reveals
	three possible cases for the $X_i$. Namely, 
	on reordering the $X_i$ if necessary, there exists $0 \leq r \leq s \leq t$ such that the following hold.
	\begin{itemize}
		\item If $1 \leq i \leq r$, then $X_i$ is a smooth hypersurface of degree $d$
			in $\PP^{m_i-1}_k$.
		\item If $r < i \leq s$, then $X_i \cong \PP^{m_i-1}_k$.
		\item If $s < i \leq t$, then $X_i = \emptyset$.
	\end{itemize}
	
	Assume now that $\sigma$ acts trivially on $\mathrm{H}^n(X_{\bar{k},\et},\QQ_\ell)$.
	Then $\tr(\sigma^\ast , \mathrm{H}^*(X_{\bar{k},\et},\QQ_\ell))$ equals 
	the  $\ell$-adic Euler characteristic  		
	$\chi(X)$ of 	$X$.
	The Lefschetz trace formula 
	(see e.g. (III.$4.11.4)$, p.$111$ of \cite{SGA5})
	then implies that
	\begin{equation}\label{eqn:ltf}
	\chi(X) = \tr(\sigma^\ast  ,  \mathrm{H}^\ast (X_{\bar{k},\et},\QQ_\ell))  = \chi(X^\sigma).
	\end{equation}
	We will derive a contradiction using the well-known formula
	\begin{equation} \label{eqn:chi}
		 \chi(X) = n + 2 + \frac{(1- d)^{n+2} -1}{d},
	\end{equation}
	which can be found in the discussion following \cite[Cor.~2.5]{Chenevert}, for example.
	Applying the Lefschetz trace formula (\ref{eqn:ltf}) and using (\ref{eqn:chi})  we obtain
	\begin{align*}
	n+ 2 + \frac{(1-d)^{n+2} -1}{d} 
	 = \sum_{i=1}^r\left( m_i +\frac{(1- d)^{m_i}  -1}{d}\right)
	+\sum_{i=r+1}^s m_{i}.
	\end{align*}
	Using \eqref{eqn:m_i} and rearranging, we find that
	\begin{equation} \label{eqn:hypersurface}
		(1-d)^{n+2} = 1 - r - d(t-s) + \sum_{i=1}^r (1-d)^{m_i}.
	\end{equation}
	We shall use this formula to obtain a contradiction.
	We first consider some special cases. If $n=1$ then the Lefschetz trace formula yields
	$$\# X^\sigma = 2 - (d-1)(d-2),$$
	which is a contradiction, as we assume that $d > 3$ when $n=1$.
	If $t=s$ and $r=0$ then \eqref{eqn:hypersurface} cannot
	be satisfied, since $d \geq 3$. 	
	
	Consider now the remaining case where $r + (t-s) \geq 1$ and $n \geq 2$.
	Note that $t \geq 2$ as $\sigma$ is non-trivial.
	We find that
	\begin{align} 
		|1 - r - d(t-s) +& \sum_{i=1}^r (1-d)^{m_i}| \nonumber \\
		& \leq (t-s) +r- 1 + (d-1)(t-s) + \sum_{i=1}^r (d-1)^{m_i} \nonumber  \\
		& \leq t -1 + \sum_{i=1}^{t} (d-1)^{m_i}. \label{eqn:easy}
	\end{align}	
	To proceed, we require the following elementary inequalities.
	\begin{align}
		n +2 \leq (1/4)\cdot x^{n+2} \quad &\mbox{ for }n \geq 2 \mbox{ and } x\geq 2. \label{eqn:1} \\
		\sum_{i=1}^t x^{m_i} \leq (3/4)\cdot x^{n+2} \quad 
		&\mbox{ for }n \geq 1, m_i\geq 1, t\geq 2 \mbox{ and } x\geq 2. \label{eqn:2}
	\end{align}
	The inequality \eqref{eqn:1} is trivial. For \eqref{eqn:2}, on using \eqref{eqn:m_i} we have
	\begin{align*}
		\sum_{i=1}^t x^{m_i} &= x^{n+2}\left(\frac{1}{x^{m_2 + \cdots +m_t}} + \cdots +
		\frac{1}{x^{m_1 + \cdots + m_{t-1}}}\right) \\
		&\leq x^{n+2}\left(\frac{1}{2^{m_2 + \cdots + m_t}} + \cdots +
		\frac{1}{2^{m_1 + \cdots + m_{t-1}}}\right).
	\end{align*}
	For $t\geq 3$ we obtain 
	$$\frac{1}{2^{m_2 + \cdots + m_t}} + \cdots +	\frac{1}{2^{m_1 + \cdots + m_{t-1}}}
	\leq \frac{1}{2^{t-1}} + \cdots +	\frac{1}{2^{t-1}} = \frac{t}{2^{t-1}} \leq \frac{3}{4},$$
	as required. For $t=2$,  using \eqref{eqn:m_i} we find that there is some $i$ for which
	$m_i \geq 2$ (as $n\geq 1$),
	which also yields \eqref{eqn:2}, as required. 
	
	Applying \eqref{eqn:1} and \eqref{eqn:2}
	to \eqref{eqn:easy}  we find that
	\begin{align*} 
		|1 - r - d(t-s) + \sum_{i=1}^r (1-d)^{m_i}| 
		& \leq (1/4)\cdot (d-1)^{n+2} -1 +(3/4)\cdot (d-1)^{n+2} \\
		& < |1-d|^{n+2},
	\end{align*}	
	which contradicts \eqref{eqn:hypersurface}.
	This completes the proof.
\end{proof}

\begin{proof}[Proof of Theorem \ref{thm:ind_auts}]
If $(d,n) \neq (4,2)$ then this  follows immediately from Proposition \ref{prop:ind_auts}, as we have
$\Lin(X) = \Aut(X)$ in such cases. So let $(d,n) = (4,2)$ and let $\sigma \in \Aut X$ act trivially on $\mathrm{H}^2(X_{\bar k,\et},\QQ_\ell)$. Then it also acts trivially on $\Pic X_{\bar{k}}$, as the cycle class map is injective here. Hence $\sigma \in \Lin(X)$, and so the result again follows from  Proposition \ref{prop:ind_auts}. 
\end{proof}

We obtain the following corollary, which will be required for attaching level structure. 

\begin{corollary}\label{cor:ind_aut} Let $d\geq 3$ and $n\geq 1$ be integers such that $(d,n) \neq (3,1)$. 
Let $k$ be a field and let $N\geq 3$ be an integer which is invertible in $k$. Suppose that $\mathrm{char}(k) = 0$ 
or $\mathrm{char}(k)>0$ is coprime to $a_{d,n}$.
If  $X$ is a smooth hypersurface of degree $d$ in $\mathbb P^{n+1}_k$, then the homomorphism \[\Aut(X) \to \Aut(\mathrm{H}^n(X_{\bar k,\et},\ZZ/N\ZZ)) \] is injective.
\end{corollary}

\begin{proof} We may assume that $N$ is a power of some prime number $\ell$. Let $\sigma \in \Lin(X)$ be a linear automorphism which acts trivially on $\mathrm{H}^n(X_{\bar k, \et},\ZZ/N \ZZ)$. Consider the action of $\sigma$ on $\mathrm{H}^n(X_{\bar k,\et},\ZZ_\ell)$, which we view as given by some $\ZZ_\ell$-matrix $A$ such that $A \bmod N$ is the identity matrix. As $\Lin(X)$ is finite  and $\mathrm{H}^n(X_{\bar k,\et},\ZZ_\ell)$ is torsion free, we see that $A$ is semi-simple and that the eigenvalues of $A \otimes \bar{\ZZ}_\ell$ are roots of unity which are equal to $1 \bmod N$, as $A \bmod N$ is the identity matrix. Therefore, as $N \geq 3$, a lemma of Minkowski and Serre (see the appendix of \cite{Serre} or the more general  \cite[Thm.~6.7]{SZ96}) implies that each eigenvalue of $A \otimes \bar{\ZZ}_\ell$ is in fact equal to $1$, and hence $A$, being semi-simple, is the identity matrix. Thus $\sigma$ acts trivially on $\mathrm{H}^n(X_{\bar k,\et},\ZZ_\ell)$. However, by our assumptions, the order of  $\sigma$ is coprime to the characteristic of $k$ and hence $\sigma$ is trivial by Theorem \ref{thm:ind_auts}, as required.
\end{proof}

\begin{remark}\label{remark}
Throughout this section we excluded the necessary case $(d,n) = (3,1)$. Indeed, let $\ell$ be a prime and let $k$ be an algebraically closed field with $\ell\in k^\ast$. If $X$ is a smooth cubic in $\mathbb P^2_k$, then  \[\ker(\Lin(X) \to \Aut(\mathrm{H}^1(X_{\et},\QQ_\ell))) \cong J[3](k),  \]
where $J[3]$ denotes the $3$-torsion group of the Jacobian $J$ of $X$. This kernel is non-trivial if and only if either $\chr(k) \neq 3$ or $\chr(k) = 3$ and $J$ is ordinary. In particular, the hypothesis $(d,n) \neq (3,1)$ can not be omitted in Theorem \ref{thm:ind_auts}. Quadric hypersurfaces are also easily seen not to satisfy the conclusion of Proposition \ref{prop:ind_auts}.
\end{remark}

\begin{remark}
The proof of Proposition \ref{prop:ind_auts} breaks down for wild automorphisms, i.e.~those automorphisms $\sigma$ whose order is divisible by $\mathrm{char}(k)$. Here $X^{\sigma}$ need no longer be smooth, hence the Lefschetz trace formula \eqref{eqn:ltf} does not take such a simple form (see \cite[\S III.4.11]{SGA5}). The tame case is however sufficient for the application to level structures for hypersurfaces.
\end{remark}

\section{Level structure} \label{sec:level}

In this section we give the promised application to the moduli space of smooth hypersurfaces with level structure.

\subsection{The stack of smooth hypersurfaces} Let $d \geq 3$ and $n\geq 1$.
Let $\mathrm{Hilb}_{d,n}$ be the Hilbert scheme of degree $d$ smooth hypersurfaces in $\mathbb P^{n+1}$ over $\ZZ$. There is a natural  left $\mathrm{PGL}_{n+2}$-action on $\mathrm{Hilb}_{d,n}$. The following quotient stack $$\mathcal C_{d,n}~:=~[\mathrm{PGL}_{n+2}\backslash \mathrm{Hilb}_{d,n}]$$ is the (moduli) stack of smooth hypersurfaces of degree $d$ in $\mathbb P^{n+1}$. (We   use \cite[\S 2.4.2]{LMB} as our main reference for quotient stacks. Note that, however, we consider left actions, whereas \emph{loc.~cit.} considers right actions.)
  A precise description of the functor of points of $\mathcal C_{d,n}$ is given in \cite[\S 2.3.2]{BenoistThesis}. For a smooth hypersurface $X$ of degree $d$ in $\mathbb P^{n+1}$ over an algebraically closed field $k$, the group $\Lin(X)$ is the group of $k$-points of the inertia group scheme of the corresponding object in $\mathcal C_{d,n}(k)$.

It follows from the arguments given in \cite[Cor.~2.5 and Prop.~4.2]{GIT} and \cite{Seshadri} that $\mathcal C_{d,n}$ is a smooth finite type separated algebraic stack with finite diagonal over $\mathbb Z$ whose coarse moduli space $\mathcal C_{d,n}^{\mathrm{coarse}}$ is an affine scheme. Moreover, if $(d,n) \neq (3,1)$, then $\mathcal C_{d,n}$ is a Deligne-Mumford stack over $\mathbb Z$; see  \cite[Thm.~1.6]{Ben13}.

\subsection{Level structure}

Let $N\geq 1$ be a positive integer.
Let  $\mathcal C_{d,n}^{[N]}$ be the category fibred in groupoids over $\ZZ[1/N]$ whose objects are triples $(S,f:X\to S, \phi)$, where $S$ is a scheme over $\ZZ[1/N]$, $f:X\to S$ is an object of $\mathcal C_{d,n}(S)$ and $\phi: R^n_{\et} f_\ast (\ZZ/N\ZZ) \to (\ZZ/N\ZZ)_S^{b_{d,n}}$ is an isomorphism of constructible sheaves on $S$. Here $b_{d,n}$ denotes the $n$th Betti number of some (hence any) smooth hypersurface of degree $d$ in $\PP^{n+1}$.  A morphism from a triple $(S,f:X\to S,\phi)$ to a triple $(S^\prime, f^\prime:X^\prime\to S^\prime,\phi^\prime)$ is defined to be a pair $(\rho, \varphi)$, where:
\begin{itemize}
 \item $\rho$ is a morphism from $(S, f:X\to S)$ to $(S',f':X'\to S')$ in $\mathcal C_{d,n}$, and
 \item $\varphi $ is an isomorphism of sheaves from $R^n_{\et} f_\ast(\ZZ/N\ZZ)$ to the pull-back of the sheaf $R^n_{\et} f'_\ast(\ZZ/N\ZZ)$ to $S$ which respects   the isomorphisms $$\phi:R^n_{\et} f_\ast (\ZZ/N\ZZ) \to (\ZZ/N\ZZ)_S^{b_{d,n}} \quad  \textrm{and} \quad \phi':R^n_{\et} f'_\ast (\ZZ/N\ZZ) \to (\ZZ/N\ZZ)_{S'}^{b_{d,n}}.$$
 \end{itemize} 
 
We call $\mathcal C_{d,n}^{[N]}$ the stack of smooth $n$-dimensional hypersurfaces of degree $d$ with \emph{level $N$ structure}.

\begin{definition}\label{defn:adn}
For $d\geq 3$ and $n\geq 1$, let $A_{d,n}$  be the set of prime numbers $p$ such that there exist an algebraically closed field $k$ and a smooth hypersurface $X$ of degree $d$ in $\mathbb P^{n+1}_k$ with a linear automorphism of order $p$. As $\mathcal C_{d,n}$ is a finite type   algebraic stack over $\ZZ$ with   finite diagonal, a standard stratification argument shows that the set $A_{d,n}$ is finite. Define $$a_{d,n} := \prod_{p\in A_{d,n}} p.$$  Note that, if $N$ is coprime to $ a_{d,n}$, then  $\mathcal C_{d,n,\ZZ[1/N]}$ is a tame stack   \cite[Thm.~3.2]{AOVTame}. 
\end{definition}
 
 We now show that, for $N\geq 3$ coprime to $ a_{d,n}$ and $(d,n) \neq (3,1)$, the stack  $\mathcal C_{d,n}^{[N]}$ is in fact an affine scheme over $\mathbb Z[1/N]$.

\begin{theorem}\label{thm:level} Let $N\geq 1$ be an integer. Suppose that $(d,n)\neq (3,1)$.
\begin{enumerate}
\item The stack $\smash{\mathcal C_{d,n}^{[N]}}$ is a $\mathrm{GL}_{b_{d,n}}(\ZZ/N\ZZ)$-torsor over $\mathcal C_{d,n,\ZZ[1/N]}$.
\item  The stack $\mathcal C_{d,n}^{[N]}$ is smooth  finite type separated and Deligne-Mumford over $\ZZ[1/N]$ with an affine coarse moduli space.
\item If $N\geq 3$ is coprime to $a_{d,n}$, then the stack \smash{$\mathcal C_{d,n}^{[N]} $} is representable by a smooth   affine scheme over $\ZZ[1/N]$. 
\end{enumerate} 
\end{theorem}
\begin{proof}
	The structure as a $\mathrm{GL}_{b_{d,n}}(\ZZ/N\ZZ)$-torsor is given by the (representable) forgetful morphism $\smash{\mathcal C_{d,n}^{[N]}}\to \mathcal C_{d,n,\ZZ[1/N]}$. This proves $(1)$. 

Since $d\geq 3 $, $n\geq 1$ and $(d,n) \neq (3,1)$, the stack $\mathcal C_{d,n}$ is a smooth finite type Deligne-Mumford stack over $\ZZ$. As $\mathcal C_{d,n}^{[N]}\to \mathcal C_{d,n,\ZZ[1/N]}$ is a representable \'etale finite type morphism, we may pull-back an \'etale finite type presentation of $\mathcal C_{d,n, \ZZ[1/N]}$ to  $\mathcal C_{d,n}^{[N]}$ to find that  $\mathcal C_{d,n}^{[N]}$ is a smooth finite type Deligne-Mumford stack. Furthermore, since  $\mathcal C_{d,n}$ is a separated algebraic stack  and the morphism  $\smash{\mathcal C_{d,n}^{[N]}}\to \mathcal C_{d,n,\ZZ[1/N]}$ is finite,  the stack $\smash{\mathcal C_{d,n}^{[N]}}$ is  separated over $\ZZ[1/N]$ with finite inertia. In particular, by \cite{KeelMori}, the stack $\smash{\mathcal C_{d,n}^{[N]}}$ has a coarse moduli space, say \smash{$\mathcal C_{d,n}^{[N]}\to \mathcal C^{[N], \textrm{co}}_{d,n}$}. Let $\mathcal {C}^{\textrm{co}}_{d,n,\ZZ[1/N]}$ be the coarse moduli space of $\mathcal C_{d,n, \ZZ[1/N]}$. Since  $\smash{\mathcal C_{d,n}^{[N]}}\to~\mathcal C_{d,n,\ZZ[1/N]}$ is finite, the induced morphism  $\smash{\mathcal C^{[N], \textrm{co}}_{d,n}}~\to~\mathcal C^{\textrm{co}}_{d,n,\ZZ[1/N]}$ is finite. As $\mathcal C^{\textrm{co}}_{d,n,\ZZ[1/N]}$ is affine, it follows that \smash{$\mathcal C^{[N], \textrm{co}}_{d,n}$} is affine \cite[Thm.~A.2]{LMB}. This proves   $(2)$.
	
Now, to prove $(3)$, let  $N\geq 3$ be coprime to $a_{d,n}$. Note that   Corollary \ref{cor:ind_aut} implies that the geometric points of the stack $\smash{\mathcal C_{d,n}^{[N]}}$ have trivial automorphism groups. In particular, by \cite[Cor.~2.2.5.(1)]{Conrad},  the stack $\smash{\mathcal C_{d,n}^{[N]}}$ is an  algebraic space over $\ZZ[1/N]$. 
Thus, the coarse moduli space morphism  $\mathcal C_{d,n}^{[N]}\to  \mathcal C_{d,n}^{[N],\textrm{co}}$ is an isomorphism. This concludes the proof. 
\end{proof}

\begin{proof}[Proof of Theorem \ref{theorem2}]  
By Theorem \ref{thm:level}, it suffices to treat the case $(d,n) = (3,1)$. 
The stack $\mathcal{C}_{3,1, \ZZ[1/3]}$ is a smooth finite type separated Deligne-Mumford  stack over $\mathbb Z[1/3]$ with affine coarse moduli space. Let $\mathcal{M}$ be the stack of elliptic curves over $\mathbb Z[1/3]$, and let $\mathcal E\to \mathcal M$ be the universal family over $\mathcal M$. Let $U$ be an affine scheme over $\ZZ[1/3]$ equipped with a finite \'etale morphism $U\to \mathcal M$ (for instance, we may take $U = \mathcal Y(3)$ to be the fine moduli scheme of elliptic curves with full level $3$ structure over $\ZZ[1/3]$; see \cite[Cor.~4.7.2]{KatzMazur}). 

 By \cite[Prop.~6.1 and Prop.~6.4]{Bergh},  the morphism of stacks $\mathcal{C}_{3,1, \ZZ[1/3]}  \to \mathcal{M} $ induced by the Jacobian is a neutral gerbe  for the finite \'etale group scheme  $\mathcal E[3]$ of $3$-torsion points of $\mathcal E\to \mathcal M$ over $\ZZ[1/3]$, i.e.~$\mathcal{C}_{3,1, \ZZ[1/3]} \cong B(\mathcal E[3])$ as stacks over $\mathcal{M}$.
By pull-back, we find that the stack $\mathcal C_{3,1,\ZZ[1/3]}$ is uniformisable by the classifying stack $B(\mathcal E[3]_U)$ over $U$. We summarise these maps in the following diagram
\[
\xymatrix{  B(\mathcal E[3]_U) \ar[d]_{\textrm{neutral } \mathcal E[3]-\textrm{gerbe}} \ar[rr]_{\textrm{finite \'etale}} & & \mathcal C_{3,1,\ZZ[1/3]} \ar[d]^{\textrm{neutral } \mathcal E[3]-\textrm{gerbe}} 	\\ U \ar[rr]^{\textrm{finite \'etale}} & & \mathcal M.    }
\]
Since $\mathcal E[3]_U\to U$ is a finite \'etale group scheme, there is a natural finite \'etale morphism $U\to B(\mathcal E[3]_U)$. In particular, the composition $U\to B(\mathcal E[3]_U) \to \mathcal C_{3,1,\ZZ[1/3]}$ is finite \'etale, and hence $\mathcal{C}_{3,1, \ZZ[1/3]}$ is uniformisable by  $U$.
Thus, for all $N\geq 1$,  the stack   $\mathcal C_{3,1, \ZZ[1/(3N)]}$ is uniformisable by a smooth affine scheme over $\ZZ[1/3N]$. As $a_{3,1}$ is divisible by $6$, this concludes the proof.
\end{proof}
 
 \begin{remark}
 	Note that there are one-dimensional smooth finite type separated Deligne-Mumford stacks  over $\mathbb C$ which are not uniformisable; see \cite{BN}.
 \end{remark}

\section{A Torelli theorem}
We now  prove Theorem \ref{thm:Torelli}. By \cite[Thm.~B]{Ach}, the intermediate Jacobian gives rise to a morphism of stacks
$$J: \mathcal{C}_{3,3, \ZZ[1/2]} \to \mathcal{A}_{5,\ZZ[1/2]},$$
where $\mathcal{A}_{5}$ denotes the stack of principally polarised abelian fivefolds. Theorem \ref{thm:Torelli} follows immediately from the following more general result.

\begin{theorem}
	The morphism of stacks
	$$J: \mathcal{C}_{3,3, \ZZ[1/2]} \to \mathcal{A}_{5,\ZZ[1/2]},$$
	is separated, representable by schemes, and universally injective.
\end{theorem}
\begin{proof}
The proof is similar to the proof of \cite[Prop.~3.2]{JL}.
The separatedness of $J$ follows from a simple application of \cite[Thm.~1.7]{Ben13} (cf.~\cite[Lem.~2.6]{JL}). Let $k$ be a field of characteristic not equal to $2$ and let $\ell \in k^*$.
The intermediate Jacobian of a cubic threefold $X$ over $k$ comes with a canonical isomorphism $\mathrm{H}^3(X_{\bar k,\et},\QQ_\ell(1)) \cong \mathrm{H}^1(J(X)_{\bar k, \et},\QQ_\ell)$ \cite[Prop.~3.6(a)]{Ach}. As the construction of the intermediate Jacobian is functorial and $\Aut J(X)$ acts faithfully on $\mathrm{H}^1(J(X)_{\bar k, \et},\QQ_\ell)$, we deduce from \cite[Thm.~1.2]{Pan} that the natural map $\Aut X \to \Aut J(X)$ is injective. As in the proof of \cite[Prop.~3.1]{JL}, we then use \cite[Tag 04Y5]{stacks-project} and \cite[Cor.~2.2.7]{Conrad} to see that $J$ is representable by algebraic spaces. However, Beauville's Torelli theorem \cite{Beau82} implies that $J$ is injective on geometric points. That $J$ is representable by schemes now follows from   \cite[Thm.~A.2]{LMB}. Finally, as $J$ is injective on geometric points and representable by schemes, it follows from \cite[Tag 03MU]{stacks-project} that $J$ itself is universally injective.
\end{proof}

\bibliography{refslevel}{}

\def\cprime{$'$}
\begin{thebibliography}{10}

\bibitem{AOVTame}
D.~Abramovich, M.~Olsson, and A.~Vistoli.
\newblock Tame stacks in positive characteristic.
\newblock {\em Ann. Inst. Fourier (Grenoble)}, 58(4):1057--1091, 2008.

\bibitem{Ach}
J.~Achter.
\newblock Arithmetic {T}orelli maps for cubic surfaces and threefolds.
\newblock {\em Trans. Amer. Math. Soc.}, 366(11):5749--5769, 2014.

\bibitem{Beau82}
A.~Beauville.
\newblock Les singularit\'es du diviseur {$\Theta $} de la jacobienne
  interm\'ediaire de l'hypersurface cubique dans {${\bf P}^{4}$}.
\newblock In {\em Algebraic threefolds ({V}arenna, 1981)}, volume 947 of {\em
  Lecture Notes in Math.}, pages 190--208. Springer, Berlin-New York, 1982.

\bibitem{Beauville2}
A.~Beauville.
\newblock Some remarks on {K}\"ahler manifolds with {$c\sb{1}=0$}.
\newblock In {\em Classification of algebraic and analytic manifolds ({K}atata,
  1982)}, volume~39 of {\em Progr. Math.}, pages 1--26. Birkh\"auser Boston,
  Boston, MA, 1983.

\bibitem{BN}
K.~Behrend and B.~Noohi.
\newblock Uniformization of {D}eligne-{M}umford curves.
\newblock {\em J. reine angew. Math.}, 599:111--153, 2006.

\bibitem{BenoistThesis}
O.~Benoist.
\newblock Espace de modules d'intersections compl\`etes lisses.
\newblock {\em Ph.D. thesis}.

\bibitem{Ben13}
O.~Benoist.
\newblock S\'eparation et propri\'et\'e de {D}eligne-{M}umford des champs de
  modules d'intersections compl\`etes lisses.
\newblock {\em J. Lond. Math. Soc. (2)}, 87(1):138--156, 2013.

\bibitem{Bergh}
Daniel Bergh.
\newblock Motivic classes of some classifying stacks.
\newblock {\em J. Lond. Math. Soc. (2)}, 93(1):219--243, 2016.

\bibitem{Boissiere}
S.~Boissi{\`e}re, M.~Nieper-Wi{\ss}kirchen, and A.~Sarti.
\newblock Higher dimensional {E}nriques varieties and automorphisms of
  generalized {K}ummer varieties.
\newblock {\em J. Math. Pures Appl. (9)}, 95(5):553--563, 2011.

\bibitem{BS67}
E.~Bombieri and H.~P.~F. Swinnerton-Dyer.
\newblock On the local zeta function of a cubic threefold.
\newblock {\em Ann. Scuola Norm. Sup. Pisa (3)}, 21:1--29, 1967.

\bibitem{C3}
J.-X. Cai, W.~Liu, and L.~Zhang.
\newblock Automorphisms of surfaces of general type with {$q\geq 2$} acting
  trivially in cohomology.
\newblock {\em Compos. Math.}, 149(10):1667--1684, 2013.

\bibitem{CPZ15}
X.~Chen, X.~Pan, and D.~Zhang.
\newblock Automorphism and cohomology {II}: Complete intersections.
\newblock {\em \texttt{arXiv:1511.07906}}.

\bibitem{Chenevert}
G.~Ch{\^e}nevert.
\newblock Representations on the cohomology of smooth projective hypersurfaces
  with symmetries.
\newblock {\em Proc. Amer. Math. Soc.}, 141(4):1185--1197, 2013.

\bibitem{CG72}
C.~H. Clemens and P.~A. Griffiths.
\newblock The intermediate {J}acobian of the cubic threefold.
\newblock {\em Ann. of Math. (2)}, 95:281--356, 1972.

\bibitem{Conrad}
B.~Conrad.
\newblock Arithmetic moduli of generalized elliptic curves.
\newblock {\em J. Inst. Math. Jussieu}, 6(2):209--278, 2007.

\bibitem{CGP}
B.~Conrad, O.~Gabber, and G.~Prasad.
\newblock {\em Pseudo-reductive groups}, volume~17 of {\em New Mathematical
  Monographs}.
\newblock Cambridge University Press, Cambridge, 2010.

\bibitem{Del72}
P.~Deligne.
\newblock Les intersections compl\`etes de niveau de {H}odge un.
\newblock {\em Invent. Math.}, 15:237--250, 1972.

\bibitem{SGA5}
A.~Grothendieck.
\newblock {\em Cohomologie {$l$}-adique et fonctions {$L$} (SGA 5)}.
\newblock Lecture Notes in Mathematics, Vol. 589. Springer-Verlag, Berlin-New
  York, 1977.
\newblock S{\'e}minaire de G{\'e}ometrie Alg{\'e}brique du Bois-Marie
  1965--1966.

\bibitem{JL}
A.~Javanpeykar and D.~Loughran.
\newblock Complete intersections: {M}oduli, {T}orelli, and good reduction.
\newblock {\em Math. Ann, to appear. \texttt{arXiv:1505.02249}}.

\bibitem{KatzMazur}
N.~M. Katz and B.~Mazur.
\newblock {\em Arithmetic moduli of elliptic curves}, volume 108 of {\em Annals
  of Mathematics Studies}.
\newblock Princeton University Press, Princeton, NJ, 1985.

\bibitem{KeelMori}
S.~Keel and S.~Mori.
\newblock Quotients by groupoids.
\newblock {\em Ann. of Math. (2)}, 145(1):193--213, 1997.

\bibitem{LMB}
G.~Laumon and L.~Moret-Bailly.
\newblock {\em Champs alg\'ebriques}, volume~39 of {\em Ergebnisse der
  Mathematik und ihrer Grenzgebiete. 3. Folge. A Series of Modern Surveys in
  Mathematics}.
\newblock Springer-Verlag, Berlin, 2000.

\bibitem{MoretBailly}
L.~Moret-Bailly.
\newblock Pinceaux de vari\'et\'es ab\'eliennes.
\newblock {\em Ast\'erisque}, (129), 1985.

\bibitem{MukaiNamikawa}
S.~Mukai and Y.~Namikawa.
\newblock Automorphisms of {E}nriques surfaces which act trivially on the
  cohomology groups.
\newblock {\em Invent. Math.}, 77(3):383--397, 1984.

\bibitem{GIT}
D.~Mumford, J.~Fogarty, and F.~Kirwan.
\newblock {\em Geometric invariant theory}, volume~34 of {\em Ergebnisse der
  Mathematik und ihrer Grenzgebiete (2)}.
\newblock Springer-Verlag, Berlin, third edition, 1994.

\bibitem{Noohi}
B.~Noohi.
\newblock Fundamental groups of algebraic stacks.
\newblock {\em J. Inst. Math. Jussieu}, 3(1):69--103, 2004.

\bibitem{Pan}
X.~Pan.
\newblock Automorphism and cohomology {I}: Fano variety of lines and {C}ubic.
\newblock {\em \texttt{arXiv:1511.05272}}.

\bibitem{Peters}
C.~A.~M. Peters.
\newblock Holomorphic automorphisms of compact {K}\"ahler surfaces and their
  induced actions in cohomology.
\newblock {\em Invent. Math.}, 52(2):143--148, 1979.

\bibitem{PoppI}
H.~Popp.
\newblock On moduli of algebraic varieties. {I}.
\newblock {\em Invent. Math.}, 22:1--40, 1973/74.

\bibitem{Rizov}
J.~Rizov.
\newblock Moduli stacks of polarized {$K3$} surfaces in mixed characteristic.
\newblock {\em Serdica Math. J.}, 32(2-3):131--178, 2006.

\bibitem{Serre}
J.-P. Serre.
\newblock Rigidit\'e du foncteur de {J}acobi d'\'echelon $n\geq 3$.
\newblock {\em Appendix of Exp. 17 of S\'eminaire Cartan}, 1960.

\bibitem{Seshadri}
C.~S. Seshadri.
\newblock Geometric reductivity over arbitrary base.
\newblock {\em Adv. in Math.}, 26(3):225--274, 1977.

\bibitem{SZ96}
A.~Silverberg and Yu.~G. Zarhin.
\newblock Variations on a theme of {M}inkowski and {S}erre.
\newblock {\em J. Pure Appl. Algebra}, 111(1-3):285--302, 1996.

\bibitem{stacks-project}
The {Stacks Project Authors}.
\newblock \emph{{S}tacks {P}roject}.
\newblock http://stacks.math.columbia.edu, 2015.

\end{thebibliography}
\bibliographystyle{plain}

\end{document}